\documentclass[amssymb,twoside,12pt]{article}
\usepackage{latexsym,amsmath,graphicx,mathrsfs,amssymb}
\usepackage{amsfonts}
\usepackage{enumitem}
\newtheorem{theorem}{Theorem}[section]
\newtheorem{lemma}[theorem]{{\bf Lemma}}

\newtheorem{rem}[theorem]{{\bf Remark}}

\newtheorem{definition}{Definition}[section]
\numberwithin{equation}{section}
\newenvironment{proof}{\indent{\em Proof:}}{\quad \hfill
$\Box$\vspace*{2ex}}

\newcommand{\bb}{{\mathfrak B}}

\begin{document}
\setcounter{page}{1}
\begin{center}
\vspace{0.3cm} {\large{\bf On the Nonlinear Impulsive\\ Volterra-Fredholm Integrodifferential Equations }} 
\\
\vspace{0.4cm}
  Pallavi U. Shikhare $^{1}$ \\
 jananishikhare13@gmail.com  \\
 
 \vspace{0.30cm}
 Kishor D. Kucche $^{2}$\\
 kdkucche@gmail.com\\
 
 \vspace{0.30cm}
  J. Vanterler da C. Sousa  $^{3}$\\
 ra160908@ime.unicamp.br\\
 \vspace{0.35cm}
 
 \vspace{0.30cm}
 $^{1,2}$ Department of Mathematics, Shivaji University, Kolhapur-416 004, Maharashtra, India.\\
 $^{3}$  Department of Applied Mathematics, Imecc-Unicamp, 13083-859, Campinas, SP, Brazil.
\end{center}

\def\baselinestretch{1.0}\small\normalsize
\begin{abstract}
In this paper, we investigate existence and uniqueness of solutions  of nonlinear Volterra-Fredholm impulsive integrodifferential equations. Utilizing theory of Picard operators we examine  data dependence of  solutions  on initial conditions and on nonlinear functions involved in integrodifferential equations. Further, we extend the integral inequality for piece-wise continuous functions to mixed case and apply it to investigate the dependence of solution on initial data through $\epsilon$-approximate solutions. It is seen that the  uniqueness and dependency results got by means of  integral inequity requires less restrictions on the functions involved in the equations than that required through  Picard operators  theory.
\end{abstract}
\allowdisplaybreaks

\section{Introduction}

Numerous evolution processes are described through the specific snapshots of time as they experience a difference in state unexpectedly. In such a cases span may be irrelevant in correlation with the length of the process. It is expected that in such cases these perturbations  act instantaneously, means in the form of impulses.  

Different issues  of the theoretical and practical importance lead us to consider the evolution of real processes with short-term perturbations. Such process are often described in the frameworks of differential and integrodifferential equations with impulse effect \cite{Bainov1,Samoilenko}. It is seen that \cite{Liu1} the differential equations with impulse conditions are commonly used to model the phenomena that cannot be modeled by the conventional initial value problems.

In the perspective on its application the  differential and integrodifferential  equations with impulse effect have been analyzed by various scientist for existence, uniqueness, stability and different types data dependency by using various techniques  \cite{Wang1,Donal,Wang2,Haoa,Zada,Zhang,Kucche0} and the references cited therein. 

Frigon and O'regan \cite{Frigon}, using the fixed point approach proved  existence results  for  impulsive initial value problem
\begin{align*} 
& w'(\tau)=f\left(\tau, w(\tau) \right),\, 0< \tau <b,\,~\tau \not= \tau_{k},\\
& \Delta w(\tau_{k})= I_{k}\left( w(\tau_{k})\right),\,~k=1,2,\cdots,m,\, m\in \mathbb{N},\\
& w(0)  =w_{0}
\end{align*}
and utilizing the idea of upper and lower solutions, authors have derived existence results for the boundary value problem
\begin{align*} 
& w'(\tau)=f\left(\tau, w(\tau) \right),\, 0< \tau <b,\,~\tau \not= \tau_{k},\\
& \Delta w(\tau_{k})= I_{k}\left( w(\tau_{k})\right),\,~k=1,2,\cdots,m,\, m\in \mathbb{N},\\
& w(0)  =w(b),
\end{align*}
where  $\Delta w(\tau_{k})= w(\tau^{+}_{k})-w(\tau^{-}_{k})$, where $ w(\tau^{+}_{k})= \underset{\epsilon\to 0^{+}}{\lim} w(\tau_{k}+\epsilon)$ and $ w(\tau^{-}_{k})= \underset{\epsilon\to 0^{-}}{\lim} w(\tau_{k}+\epsilon)$.

Using Picard, weakly Picard operators theory Bielecki norms, Wang et al. \cite{Wang1}, have examined nonlocal problem 
\begin{align*} 
& w'(\tau)=f\left( \tau, w(\tau)\right),\, \tau\in [0,b],\\
& w(0)= w_{0}+g(w),
\end{align*}
for existence, uniqueness and data dependence .  Authors have expanded the acquired outcomes at that point to a class of impulsive Cauchy problems by adapting the same strategies. Wang et al. \cite{Wang}, by applying the integral inequality of Gronwall type for piece-wise continuous functions investigated  Ulam--Hyers stability for impulsive ordinary differential equations. 

Liu \cite{Liu1} studied  the existence and uniqueness of mild and classical
 solutions for a nonlinear impulsive evolution equation
 \begin{align*} 
 & w'(\tau)=\mathscr{A}w(\tau)+f\left(\tau, w(\tau) \right),\,  0<\tau <b,\,~\tau \not= \tau_{k}, \\
 & \Delta w(\tau_{k})= I_{k}\left( w(\tau_{k})\right),\,~k=1,2,\cdots,~ \tau_0<\tau_1< \cdots<b\\
 & w(0)  =w_{0}
 \end{align*}
in a Banach space $X$, where $\mathscr{A}$ is the generator of a strongly continuous semigroup.

On the other hand, Anguraj et al. \cite{Anguraj}, using semigroup theory and contraction mapping principle, proved the existence and uniqueness of the mild and classical solutions for the impulsive evolution mixed Volterra-Fredholm integrodifferential equation. Muresan \cite{Muresan} explored existence, uniqueness and data dependence of the solutions to  mixed Volterra-Fredholm integrodifferential equation in Banach space by
Utilizing Picard and weakly Picard operators’ method and Bielecki norms.

It is noticed that in many of the works \cite{Rus5}-\cite{Kucche00}, differential and integral inequalities \cite{Pachpatte1,Pachpatte2}  play central role in the investigation of different properties of solution such as uniqueness, boundedness, stability etc. 

Motivated by works \cite{Wang1,Wang,Anguraj,Muresan}, we will  investigate the existence, uniqueness and continuous data dependence of solutions of nonlinear Volterra-Fredholm impulsive integrodifferential equations (VFIIDEs) of the form:
\begin{small}
\begin{align}\label{F1}
 & w'(\tau) =\mathscr Aw(\tau)+G\left(\tau, w(\tau), \int_0^\tau F_1(\tau,\sigma,w(\sigma))d\sigma,\int_0^b  F_2(\tau,\sigma,w(\sigma))d\sigma\right),\nonumber\\
&\tau \in J,\tau \not= \tau_{k},k=1,2,\cdots,n, \\
& w(0) = w_{0},\, w_{0}\in X \label{F2}\\
&\Delta w(\tau_{k})= I_{k}(w(\tau_{k})),~k=1,2,\cdots,n ,\label{0F3}
\end{align}
\end{small}
where $ J= [0,b],\,\mathscr A:X\to X$ is the infinitesimal generator of $C_0$-semigroup $\{\mathscr T(\tau)\}_{\tau\geq 0}$ in Banach space $(X,\|\cdot\|),\,I_{k}:X\to X\,(k=1,\cdots n)$ are continuous functions and  $ G, ~F_{1}\,\text{and}\,  F_{2}$ are the functions specified later. The impulsive moments $\tau_{k}$ are such that  $0\leq\tau_{0}<\tau_{1}<\cdots<\tau_{n}<\tau_{n+1}\leq b, n\in \mathbb{N}$. Further, $\Delta w(\tau_{k})= w(\tau^{+}_{k})-w(\tau^{-}_{k})$, where  $ w(\tau^{+}_{k})= \underset{h\to 0^{+}}{\lim} w(\tau_{k}+h)$  and $ w(\tau^{-}_{k})= \underset{h\to 0^{-}}{\lim} w(\tau_{k}+h)$ are respectively the right and left limits of $w$ at $\tau_{k}$.  

The  dependence of  solutions  on initial conditions is firstly obtained via Picards' operator technique.  Further, we extend  the integral inequality for piece-wise continuous functions given in Theorem 2 of \cite{Hristova} for mixed case. The  extended version of integral inequality we obtained then utilized to analyze the dependence of solution on initial data through $\epsilon$-approximate solutions. It is seen that results we obtained via integral inequity regrading uniqueness and dependence of solution requires less restrictions on the nonlinear functions involved in the equations than that are demanded through  Picard operators  theory.

This paper is organized as follows. Section 2, relates with  preliminaries. We will discuss  existence, uniqueness  and continuous data dependence  in section 3. Section 4, deals with dependency of solutions via Picard theory.  In section 5, we  prove the variant of  integral inequality for piece-wise continuous functions. In section 6,  we provide the application of integral inequality we obtained to study of  data dependence via $\epsilon$-approximate solution to VFIIDEs.  Paper finishes with concluding remarks.


\section{Preliminaries} \label{preliminaries}
\begin{definition} [\cite{Wang,Rus3,Ilea}]. 
Let $\left(X, d \right)$  be a metric space. An operator $\mathcal{A}:X \rightarrow X$ is a Picard operator {\rm (PO)}, if there exists $w^{\ast }\in X $ satisfying the following conditions:
\begin{flushleft}
\begin{tabular}{cl}
{\rm (a)} & $\mathbf{F}_{\mathcal{A}}=\{w^{\ast }\}$, where $\mathbf{F}_{\mathcal{A}}:=\{w\in X  : \mathcal A\left( w\right) =w\}$. \\ 
{\rm (b)} & the sequence $\left( \mathcal A^{n}(w_{0})\right) _{n\in \mathbb{N}}$ converges to $w^{\ast }$ for all $w_{0}\in X$.
\end{tabular}
\end{flushleft}
\end{definition}
\begin{theorem}[\cite{Wang,Rus3,Ilea}]. \label{B1}
Let $(Y,d)$ be a complete metric space and $\mathcal{A},\mathcal{B}: Y\to Y$ two operators. We suppose the following:
\begin{flushleft}
\begin{tabular}{cl}
{\rm (a)} & $\mathcal{A} $ is a contraction with contraction constant $\alpha$ and $\mathbf{F}_{\mathcal{A}}=\{w^{\ast }_{\mathcal{A}}\}$ ; \\ 
{\rm (b)} &  $\mathcal{B}$ has fixed point and $w^{\ast }_{\mathcal{B}}\in\mathbf{F}_{\mathcal{B}} $;\\
{\rm (c)} & there exists $\rho>0$ such that $d\left( \mathcal{A}(w),\mathcal{B}(w)\right)\leq\rho$ for all $w\in Y.$
\end{tabular}
\end{flushleft}
Then $$d \left( w^{\ast }_{\mathcal{A}},w^{\ast }_{\mathcal{B}}\right)  \leq \frac{\rho}{1-\alpha}.$$
\end{theorem}
\begin{lemma}[\cite{Hristova}, Theorem 16.4 ]\label{L1}. Let for $\tau\geq \tau_{0}$ the inequality
\begin{small}
\begin{align*}
u(\tau)\leq \tilde{a}(\tau)+\int_{\tau_{0}}^{\tau} g(\tau,\sigma)u(\sigma)d\sigma+\underset{\tau_{0}<\tau_{k}<\tau}{\sum} ~\tilde{\beta}_{k}(\tau)\,u(\tau_{k}),
\end{align*}
\end{small}
hold, where $~\tilde{\beta}_{k}(\tau)\, (k\in \mathbb{N})$ are nondecreasing functions for $\tau\geq \tau_{0},\,\tilde{a}\in PC([\tau_{0},\infty ),\mathbb{R}_{+})$ is  a nondecreasing function, $u\in PC([\tau_{0},\infty ),\mathbb{R}_{+})$, and $g(\tau,\sigma)$ is a continuous nonnegative function for $\tau,\sigma\geq \tau_{0}$ and nondecreasing with respect to $\tau$ for any fixed $\sigma\geq \tau_{0}$.\\
Then, for $\tau\geq \tau_{0}$ the following inequality is valid:
\begin{small}
\begin{align*}
u(\tau)\leq \tilde{a}(\tau)\underset{\tau_{0}< \tau_{k}<\tau}{\prod} \left( 1+\tilde{\beta}_{k}(\tau)\right) \exp \left( \int_{\tau_{0}}^{\tau} g(\tau,\sigma)d\sigma\right).
\end{align*}
\end{small}
\end{lemma}
\begin{theorem} [\cite{Hristova}, Theorem 2]\label{LF1}. Let for $\tau\geq \tau_{0}$ the following inequality hold
\begin{small}
\begin{align*}
u(\tau)\leq a(\tau)+\int_{\tau_{0}}^{\tau} b(\tau,\sigma)u(\sigma)d\sigma+\int_{\tau_{0}}^{\tau}\left( \int_{\tau_{0}}^{\sigma} k(\tau,\sigma,\varsigma)u(\varsigma)d\varsigma\right) d\sigma+\underset{\tau_{0}<\tau_{k}<\tau}{\sum} ~\beta_{k}(\tau)\,u(\tau_{k}),
\end{align*}
\end{small}
where $u,\,a \in PC([\tau_{0},\infty),\mathbb{R}_{+}),$ $\tilde{a}$ is a nondecreasing, $b(\tau,\sigma)$ and $\,k(\tau, \sigma, \varsigma)$ are continuous and nonnegative functions for $ \tau, \sigma, \varsigma \geq \tau_{0}$ and are nondecreasing with respect to $\tau,\,\, \beta_{k}(\tau)\,(k\in \mathbb{N})$ are nondecreasing for $\tau\geq \tau_{0}$. Then, for $\tau\geq \tau_{0}$, the following inequality is valid:
\begin{small}
\begin{align*}
u(\tau)&\leq a(\tau)\underset{\tau_{0}<\tau_{k}<\tau}{\prod}\left( 1+\beta_{k}(\tau)\right) \exp \left( \int_{\tau_{0}}^\tau b(\tau, \sigma )d\sigma+\int_{\tau_{0}}^\tau\int_{\tau_{0}}^\sigma k(\tau, \sigma, \varsigma)d\varsigma d\sigma\right).
\end{align*}
\end{small}
\end{theorem}
We need the following theorem from Pazy \cite{Pazy}. 
\begin{theorem} [\cite{{Pazy}}] \label{th2.1}. Let $\{\mathscr T(\tau)\}_{\tau\geq 0}$  is a $C_0$-semigroup. There exists constants $\omega\geq 0$ and $ \mathfrak{M}\geq 1$
such that $$\|\mathscr T(\tau)\|\leq \mathfrak{M}e^{\omega \tau},   ~0\leq \tau < \infty . $$
\end{theorem}
\section{Existence and uniqueness }

Consider the following space
 \begin{align*}
&\Theta =\left\lbrace   w:J\to X :w(\tau) ~\text{is continuous at}\,  \tau \not= \tau_{k},\text{left continuous at}\, \tau=\tau_{k}, \right.\\
& \left.\text{the right limit}~w(\tau^{+}_{k})~\text{exists for} ~k =1, \cdots n,\,n\in \mathbb{N} ~\text{and}~ w(0)=w_{0}\right\rbrace.
\end{align*}
Consider the following Banach space $\Theta_{\mathcal{PB}}=\left(\Theta,\left\| \cdot\right\|_{\mathcal{PB}} \right)$,
where 
\begin{small}
$$ \left\| w\right\|_{\mathcal{PB}}=\underset{\tau\in J}{\sup}\left\lbrace \frac{\left\| w(\tau) \right\|}{e^{\gamma \tau}} \right\rbrace , w\in \Theta,\,\gamma>0,$$
\end{small}
is the piece-wise Bielecki norm, and $\Theta_{\mathcal{PC}}=\left(\Theta,\left\| \cdot\right\|_{\mathcal{PC}} \right)$, where $\left\|w \right\|_{\mathcal{PC}}=\underset{t\in J}{\sup}\left\lbrace \left\| w(\tau) \right\| \right\rbrace,\, w\in \Theta$ is the piece-wise Chebyshev norm.

\begin{definition}
A function $ w\in \Theta $ is called a mild solution of \eqref{F1}-\eqref{0F3} if it satisfies the following  impulsive integral equation 
\begin{small}
\begin{align}\label{F0}
 w(\tau)=&\mathscr T(\tau)w_{0}+\int_0^\tau \mathscr T(\tau-\sigma)  G\left(\sigma, w(\sigma), \int_0^\sigma F_{1}(\sigma,\varsigma,w(\varsigma))d\varsigma,\int_0^b  F_{2}(\sigma,\varsigma,w(\varsigma))d\varsigma \right)d\sigma \nonumber\\
& \qquad +\underset{0<\tau_{k}<\tau}{\sum} \mathscr T(\tau-\tau_{k}) I_{k}\left(w(\tau_{k})\right),~ \tau\in J.
\end{align}
\end{small}
\end{definition}

We need the following hypothesis to obtain our main results.
\begin{itemize}
\item[(H1)]
Let $G:J\times X \times X\times X\to X$ be continuous function and there exist constant $L_{G}>0$ such that
\begin{small} 
$$ \|G(\tau, v_{1},v_{2},v_{3})-G(\tau, w_{1},w_{2},w_{3})\| \leq L_{G} \left( \sum_{i=1}^{3}\|v_{i}- w_{i}\| \right) ,$$
\end{small}
for all $ \tau \in J$ and $v_{i},w_{i}\in X \, (i=1,2,3) .$
\item[(H2)]Let $F_{j}\,\,(j=1,2): J\times J\times X\to X$  are continuous functions and there exist constants $L_{F_{j}}\,\,(j=1,2)>0$ such that
\begin{small} 
$$ \|F_{j}(\tau, \sigma,v_{1})-F_{j}(\tau,\sigma,w_{1})\| \leq L_{F_{j}} \|v_{1}- w_{1}\|,\,\, j=1,2  ;$$
\end{small}
for all $ \tau,\sigma\in J$ and $v_{1},w_{1}\in X .$
\item[(H3)]

There exist constant $L_{I_{k}} >0 $ such that $\left\| I_{k}(v)-I_{k}(w)\right\| \leq L_{I_{k}} \left\|v-w\right\|;$ for $v,w\in X,\,(k=1,\cdots, n)$.
\end{itemize}
\begin{theorem}\label{thm1}

Suppose that hypothesis {\rm (H1)-(H3)} are holds and there exist constant $\gamma>0 $ such that
\begin{small}
$$L_{\mathcal{R}}=\frac{\mathscr M\,L_{G}}{\gamma}\left[ \left( 1-e^{-\gamma\,b}\right) \left( 1+ \frac{L_{F_{1}}}{\gamma}\right) + L_{F_{2}} \,b\, e^{\gamma\,b}\right] +\mathscr M\, e^{\gamma b} \sum_{k=1}^{n}L_{I_{k}}  <1.$$
\end{small}

 Then the {\rm VFIIDE}  \eqref{F1}-\eqref{0F3} has a unique solution in $\Theta_{\mathcal{PB}}$. 
\end{theorem}
\begin{proof}

Define the operator $\mathcal{R}: \Theta_{\mathcal{PB}}\to \Theta_{\mathcal{PB}},\,\,\Theta_{\mathcal{PB}}=(\Theta, \left\| \cdot\right\|_{\mathcal{PB}})$ by 
\begin{small}
\begin{align*}
\mathcal{R}(w)(\tau)&=\mathscr T(\tau)w_{0}+\int_0^\tau \mathscr T(\tau-\sigma)  G\left(\sigma, w(\sigma), \int_0^\sigma  F_{1}(\sigma,\varsigma,w(\varsigma))d\varsigma,\int_0^b  F_{2}(\sigma,\varsigma,w(\varsigma))d\varsigma \right)d\sigma \\
& \qquad +\underset{0<\tau_{k}<\tau}{\sum} \mathscr T(\tau-\tau_{k}) I_{k}\left(w(\tau_{k})\right),~ \tau\in [0,b].
\end{align*}
\end{small}

Then fixed point of the operator $ \mathcal{R}$ is the solution of the problem \eqref{F1}-\eqref{0F3}. Let any $w,v \in \Theta $ and $\tau\in [0,b]$. Then
\begin{small}
\begin{align}\label{F3}
&\left\|\mathcal{R}(w)(\tau)-\mathcal{R}(v)(\tau) \right\|\nonumber\\
&\leq\int_0^\tau \left\| \mathscr T(\tau-\sigma)\right\|_{B(X)}   \left\| G\left(\sigma, w(\sigma), \int_0^\sigma  F_{1}(\sigma,\varsigma,w(\varsigma))d\varsigma,\int_0^b  F_{2}(\sigma,\varsigma,w(\varsigma))d\varsigma \right)\right.\nonumber\\
&\left.\qquad-G\left(\sigma, v(\sigma), \int_0^\sigma  F_{1}(\sigma,\varsigma,v(\varsigma))d\varsigma,\int_0^b  F_{2}(\sigma,\varsigma,v(\varsigma))d\varsigma \right)\right\| d\sigma \nonumber\\
& \qquad\qquad +\underset{0<\tau_{k}<\tau}{\sum}\left\|  \mathscr T(\tau-\tau_{k})\right\|_{B(X)}  \left\| I_{k}\left(w(\tau_{k})\right)-I_{k}\left(v(\tau_{k})\right)\right\|.
\end{align}
\end{small}

By the Theorem \ref{th2.1} there exist constant $\mathscr M\geq 1 $ such that
\begin{align}
\left\| \mathscr{T}(\tau)\right\| _{B(X)} \leq \mathscr M,\, \tau\geq 0. \label{F4}
\end{align}

Using hypothesis (H1)-(H3) and the condition \eqref{F4} to the inequality \eqref{F3}, we have \begin{small}
\begin{align*}
&\left\|\mathcal{R}(w)(\tau)-\mathcal{R}(v)(\tau) \right\|\nonumber\\
&\leq\int_0^\tau \mathscr M\,L_{G} \left[\left\| w(\sigma)-v(\sigma)\right\|  e^{-\gamma\,\sigma}\right] e^{\gamma\,\sigma} d\sigma +\int_0^\tau \int_0^\sigma \mathscr M\,L_{G}\,L_{F_{1}} \left[\left\| w(\varsigma)-v(\varsigma)\right\|  e^{-\gamma\,\varsigma}\right] e^{\gamma\,\varsigma} d\varsigma d\sigma  \\
&\qquad +\int_0^\tau \int_0^b \mathscr M\,L_{G}\,L_{F_{2}} \left[\left\| w(\varsigma)-v(\varsigma)\right\|  e^{-\gamma\,\varsigma}\right] e^{\gamma\,\varsigma} d\varsigma d\sigma +\underset{0<\tau_{k}<\tau}{\sum} \mathscr M\,L_{I_{k}} \left[\left\| w(\tau_{k})-v(\tau_{k})\right\|  e^{-\gamma\,\tau_{k}}\right] e^{\gamma\,\tau_{k}}\\
&\leq\int_0^\tau \mathscr M\,L_{G} \left\| w-v\right\|_{\mathcal{PB}} e^{\gamma\,\sigma} d\sigma +\int_0^\tau \int_0^\sigma \mathscr M\,L_{G}\,L_{F_{1}} \left\| w-v\right\|_{\mathcal{PB}} e^{\gamma\,\varsigma} d\varsigma d\sigma \\
&\qquad +\int_0^\tau \int_0^b \mathscr M\,L_{G}\,L_{F_{2}} \left\| w-v\right\|_{\mathcal{PB}} e^{\gamma\,\varsigma} d\varsigma d\sigma+ \mathscr M \sum_{k=1}^{n}L_{I_{k}} e^{\gamma \tau_{k}}\,L_{I_{k}} \left\| w-v\right\|_{\mathcal{PB}} \\
&= \left\lbrace \mathscr M\,L_{G} \left(\frac{e^{\gamma \tau}}{\gamma}- \frac{1}{\gamma}\right)+\mathscr M\,L_{G} L_{F_{1}} \left(\frac{e^{\gamma \tau}}{\gamma^{2}}- \frac{1}{\gamma^{2}}-\frac{\tau}{\gamma}\right)\right.\\
&\qquad \qquad \left.+\mathscr M\,L_{G} L_{F_{2}} \left(\frac{e^{\gamma b}}{\gamma}- \frac{1}{\gamma}\right)\tau+\mathscr M\,\, \sum_{k=1}^{n} e^{\gamma \tau_{k}}\,L_{I_{k}}  \right\rbrace\left\|w-v \right\|_{\mathcal{PB}}\\
& \leq  \left\lbrace \mathscr M\,L_{G} \left(\frac{e^{\gamma \tau}}{\gamma}- \frac{1}{\gamma}\right)+\mathscr M\,L_{G} L_{F_{1}} \left(\frac{e^{\gamma \tau}}{\gamma^{2}}- \frac{1}{\gamma^{2}}\right)\right. \left.+\mathscr M\,L_{G} L_{F_{2}}\, b \frac{e^{\gamma b}}{\gamma} +\mathscr M\, e^{\gamma b}  \sum_{k=1}^{n}L_{I_{k}}  \right\rbrace\left\|w-v \right\|_{\mathcal{PB}},\tau \in J.
\end{align*}
\end{small}

Thus
\begin{small}
\begin{align*}
\left\|\mathcal{R}(w)(\tau)-\mathcal{R}(v)(\tau) \right\| e^{-\gamma \tau} & \leq \left\lbrace \mathscr M\,L_{G} \frac{1-e^{-\gamma \tau}}{\gamma}+\mathscr M\,L_{G} L_{F_{1}} \frac{1-e^{-\gamma \tau}}{\gamma^{2}}\right.\\
&\qquad \qquad\left.+\mathscr M\,L_{G} L_{F_{2}}\,  \frac{b}{\gamma} e^{\gamma b} e^{-\gamma \tau}+\mathscr M\, e^{\gamma b}\,e^{-\gamma \tau}\sum_{k=1}^{n}L_{I_{k}} \right\rbrace\left\|w-v \right\|_{\mathcal{PB}}\\
& \leq \left\lbrace \mathscr M\,L_{G} \frac{1-e^{-\gamma \tau}}{\gamma}+\mathscr M\,L_{G} L_{F_{1}} \frac{1-e^{-\gamma \tau}}{\gamma^{2}}\right.\\
&\qquad \qquad\left.+\mathscr M\,L_{G} L_{F_{2}}\,  \frac{b}{\gamma} e^{\gamma b} +\mathscr M\, e^{\gamma b}\sum_{k=1}^{n}L_{I_{k}} \right\rbrace\left\|w-v \right\|_{\mathcal{PB}}.
\end{align*}
\end{small}

Therefore
\begin{small}
\begin{align*}
\left\|\mathcal{R}(w)-\mathcal{R}(v) \right\|_{\mathcal{PB}}
&=\underset{\tau\in J}{\sup}\left\lbrace \frac{\left\|\mathcal{R}(w)(\tau)-\mathcal{R}(v)(\tau) \right\|}{e^{\gamma \tau}} \right\rbrace \\
&\leq\left( \frac{\mathscr M\,L_{G}}{\gamma}\left[ \left( 1-e^{-\gamma\,b}\right) \left( 1+ \frac{L_{F_{1}}}{\gamma}\right) + L_{F_{2}} \,b\, e^{\gamma\,b}\right] +\mathscr M\, e^{\gamma b}\sum_{k=1}^{n}L_{I_{k}} \right) \left\|w-v \right\|_{\mathcal{PB}}\\
&=L_{\mathcal{R}}\left\|w-v \right\|_{\mathcal{PB}}.
\end{align*}
\end{small}

Choose $\gamma>0$ such that 
\begin{small}
$$L_{\mathcal{R}}=\frac{\mathscr M\,L_{G}}{\gamma}\left[ \left( 1-e^{-\gamma\,b}\right) \left( 1+ \frac{L_{F_{1}}}{\gamma}\right) + L_{F_{2}} \,b\, e^{\gamma\,b}\right] +\mathscr M\, e^{\gamma b}\sum_{k=1}^{n}L_{I_{k}}<1.$$ 
\end{small}
Then $\mathcal{R}$ is contraction operator. By Banach fixed point theorem it  has a fixed point $\bar{w}\in  \Theta_{\mathcal{PB}}$ which is unique solution of VFIIDE  \eqref{F1}-\eqref{0F3}.
\end{proof}

\section{Dependency of solutions via PO}
In this section, we analyse the dependency of solutions on the initial condition and functions in the equations by means of Picard operator theory. 

Consider the following problem
\begin{small}
\begin{align}\label{F5}
& w'(\tau)=\mathscr Aw(\tau)+\widehat{G}\left(\tau, w(\tau), \int_0^\tau \widehat{F}_1(\tau,\sigma,w(\sigma))d\sigma,\int_0^b  \widehat{F}_2(\tau,\sigma,w(\sigma))d\sigma\right),\nonumber\\
& \tau \in J,\tau \not= \tau_{k},k=1,2,\cdots,n,\\
& w(0) = \widehat{w}_{0},\, \widehat{w}_{0}\in X \label{F6}\\
& \Delta w(\tau_{k})=\hat{I}_{k}(w(\tau_{k})),~k=1,2,\cdots,n ,\label{0F7}
\end{align}
\end{small}
where  $ \widehat{G}:J\times X \times X\times X\to X,\,\widehat{F}_j\,(j=1,2):J\times J\times X\to X\,\text{and}\,\hat{I}_{k}:X\to X\,( k=1, \cdots n)$ are the continuous functions.

A function $ w\in \Theta $ is called a mild solution of \eqref{F5}-\eqref{0F7} if it satisfies the following  impulsive integral equation 
\begin{small}
\begin{align}\label{F00}
w(\tau)=&\mathscr T(\tau)\widehat{w}_{0}+\int_0^\tau \mathscr T(\tau-\sigma)  \widehat{G}\left(\sigma, w(\sigma), \int_0^\sigma \widehat{F}_1(\sigma,\varsigma,w(\varsigma))d\varsigma,\int_0^b  \widehat{F}_2(\sigma,\varsigma,w(\varsigma))d\varsigma \right)d\sigma \nonumber\\
& \qquad +\underset{0<\tau_{k}<\tau}{\sum} \mathscr T(\tau-\tau_{k}) \hat{I}_{k}\left(w(\tau_{k})\right),~ \tau\in [0,b].
\end{align}
\end{small}
\begin{theorem}\label{thm2}

Suppose that the following  conditions are holds
\item{\rm (A1)}
All the conditions in Theorem \ref{thm1} are satisfied and $ w^{\ast}\in \Theta$ is the unique solution of the integral equation \eqref{F0}.
\item{\rm(A2)}
There exists constants $L_{\widehat{G}}, L_{ \widehat{F}_{1}},L_{ \widehat{F}_{2}}> 0$  such that
\begin{small}
 $$ \left\|\widehat{G} (\tau,w_{1},w_{2},w_{3})-\widehat{G}(\tau,v_{1},v_{2},v_{3}) \right\|\leq L_{\widehat{G}} \left(\sum_{i=1}^{3} \left\| w_{i}-v_{i}\right\| \right) $$
\end{small}
and
\begin{small}
$$ \left\|L_{ \widehat{F}_{j}}(\tau,\sigma,w_{1})-L_{ \widehat{F}_{j}}(\tau,\sigma,v_{1}) \right\|\leq L_{L_{ \widehat{F}_{j}}} \left( \left\| w_{1}-v_{1}\right\| \right), \, j=1,2; $$
\end{small}
for all $\tau,\sigma \in J$ and $w_{i},v_{i}\in X\, (i=1,2,3)$.
\item{\rm (A3)} There exist constant $L_{\hat{I}_{k}}  $ such that $\left\| \hat{I}_{k}(w)-\hat{I}_{k}(v)\right\| \leq L_{\hat{I}_{k}} \left\|v-w\right\|;$ for $v,w\in X$.
\item{\rm(A4)} There exists a constants $\mu,\,\eta >0 $ such that
$$ \left\|G(\tau,u,v,w)-\widehat{G} (\tau,u,\tilde{v},\tilde{w}) \right\|\leq \mu  $$
and
$$\left\| I_{k}(w)-\hat{I}_{k}(w)\right\| \leq \eta ;$$ 
for all $\tau \in J$ and $u,\,v,\,\tilde{v},\,\,w,\,\tilde{w}\in X$.

Then, if $v^{\ast}$ is the solution of integral equations \eqref{F00} then 
\begin{small}
\begin{align}\label{f0}
 \left\|w^{\ast}-v^{\ast} \right\|_{\mathcal{PB}}\leq \frac{\mathscr M \left\| w_{0}-\hat{w}_{0}\right\| +b\,\mathscr M\, \mu + n\, \mathscr M\, \eta }{1-L_{\mathcal{R}}}. 
\end{align}
\end{small}
\end{theorem}
\begin{proof}

Define the  operators $\mathcal{R}, \mathcal{T}: (\Theta, \left\| \cdot\right\|_{\mathcal{PB}}) \to (\Theta, \left\| \cdot\right\|_{\mathcal{PB}})$  defined by
\begin{small}
\begin{align*}
\mathcal{R}(w)(\tau)&=\mathscr T(\tau)w_{0}+\int_0^\tau \mathscr T(\tau-\sigma)  G\left(\sigma, w(\sigma), \int_0^\sigma  F_{1}(\sigma,\varsigma,w(\varsigma))d\varsigma,\int_0^b  F_{2}(\sigma,\varsigma,w(\varsigma))d\varsigma \right)d\sigma \\
& \qquad +\underset{0<\tau_{k}<\tau}{\sum} \mathscr T(\tau-\tau_{k}) I_{k}\left(w(\tau_{k})\right)
\end{align*}
\end{small}
and
\begin{small}
\begin{align*}
\mathcal{T} w(\tau)&=\mathscr T(\tau)\widehat{w}_{0}+\int_0^\tau \mathscr T(\tau-\sigma)  \widehat{G}\left(\sigma, w(\sigma), \int_0^\sigma \widehat{F}_1(\sigma,\varsigma,w(\varsigma))d\varsigma,\int_0^b  \widehat{F}_2(\sigma,\varsigma,w(\varsigma))d\varsigma \right)d\sigma \nonumber\\
& \qquad +\underset{0<\tau_{k}<\tau}{\sum} \mathscr T(\tau-\tau_{k}) \hat{I}_{k}\left(w(\tau_{k})\right).
\end{align*}
\end{small}
With (A1), it is already prove that  $\mathcal{R}$ is a contraction. On the similar line $\mathcal{T}$ is contraction provided that $$L_{\mathcal{T}}=\frac{\mathscr M\,L_{\widehat{G}}}{\gamma}\left[ \left( 1-e^{-\gamma\,b}\right) \left( 1+ \frac{L_{\widehat{F}_1}}{\gamma}\right) + L_{\widehat{F}_2} \,b\, e^{\gamma\,b}\right] +\mathscr M\,\, e^{\gamma\,b} \sum_{k=1}^{n}L_{\hat{I}_{k}}  <1.$$  Let $\mathbf{F}_{\mathcal{R}}=\{w^{\ast}\}$ and $\mathbf{F}_{\mathcal{T}}=\{v^{\ast}\}$. For any $w \in  \Theta$. Then any $\tau\in J$, we have
\begin{small}
\begin{align*}
&\left\|\mathcal{R}(w)(\tau)-\mathcal{T}(w)(\tau) \right\| \nonumber\\
&\leq \left\| \mathscr T(\tau)\right\|_{B(X)} \left\| w_{0}-\hat{w}_{0}\right\| +\int_{0}^{\tau} \left\| \mathscr T(\tau-\sigma)\right\|_{B(X)} \left\|  G\left(\sigma, w(\sigma), \int_0^\sigma  F_{1}(\sigma,\varsigma,w(\varsigma))d\varsigma,\right.\right.\nonumber\\
& \qquad\left.\left.\int_0^b  F_{2}(\sigma,\varsigma,w(\varsigma))d\varsigma \right)- \widehat{G}\left(\sigma, w(\sigma), \int_0^\sigma \widehat{F}_1(\sigma,\varsigma,w(\varsigma))d\varsigma,\int_0^b  \widehat{F}_2(\sigma,\varsigma,w(\varsigma))d\varsigma \right)\right\| \nonumber\\
& \qquad \qquad+\underset{0<\tau_{k}<\tau}{\sum} \left\| \mathscr T(\tau-\tau_{k}) \right\|_{B(X)}\left\|I_{k}\left(w(\tau_{k})\right)-\hat{I}_{k}\left(w(\tau_{k})\right) \right\| . 
\end{align*}
\end{small}

In the view of assumptions (A4), we have 
\begin{small}
\begin{align*}
\left\|\mathcal{R}(w)(\tau)-\mathcal{T}(w)(\tau) \right\|
 \leq \mathscr M \left\| w_{0}-\hat{w}_{0}\right\| +b\,\mathscr M\, \mu + n\, \mathscr M\, \eta.
\end{align*}
\end{small}

Therefore,
\begin{small}
\begin{align}\label{F8}
\left\|\mathcal{R}(w)-\mathcal{T}(w) \right\|_{\mathcal{PB}}
&=\underset{\tau\in J}{\sup}\left\lbrace \frac{\left\|\mathcal{R}(w)(\tau)-\mathcal{T}(w)(\tau) \right\|}{e^{\gamma \tau}} \right\rbrace \nonumber\\
&\le \mathscr M \left\| w_{0}-\hat{w}_{0}\right\| +b\,\mathscr M\, \mu + n\, \mathscr M\, \eta.
 \end{align}
 \end{small}
 
Applying the Theorem \ref{B1}, to the inequality \eqref{F8}, we obtain 
$$\left\|w^{\ast}-v^{\ast} \right\|_{\mathcal{PB}}\leq \frac{\mathscr M \left\| w_{0}-\hat{w}_{0}\right\| +b\,\mathscr M\, \mu + n\, \mathscr M\, \eta }{1-L_{\mathcal{R}}},$$
which is  desired inequality \eqref{f0}.
\end{proof}

\section{Extended version of integral inequality for piece-wise continuous functions}

In this section, firstly we extended the integral inequality given in the Theorem \ref{LF1} to the mixed case, so as the results related to dependency of solutions on different data can be obtained for mixed Volterra-Fredholm  integrodifferential equations with impulses.
\begin{theorem}\label{Th3}
Let $ \tau\in[0,b],$ the following integral inequality hold
\begin{small}
\begin{align}\label{v1}
u(\tau)&\leq a(\tau)+\int_0^\tau b(\tau,\sigma)u(\sigma)d\sigma+\int_0^\tau\left( \int_0^\sigma k_{1}(\tau, \sigma, \varsigma) u(\varsigma) d\varsigma \right)d\sigma+ \int_0^\tau\left( \int_0^b k_{2}(\tau, \sigma, \varsigma) u(\varsigma) d\varsigma \right)d\sigma\nonumber\\
&\qquad+\underset{0<\tau_{k}<\tau}{\sum} ~\beta_{k}(\tau)u(\tau_{k})
\end{align}
\end{small}
where $u,a\in PC([0,b],\mathbb{R}_{+}),$ $a$ is a nondecreasing, $b(\tau,\sigma), \,k_{1}(\tau, \sigma, \varsigma)$ and $\,k_{2}(\tau, \sigma, \varsigma)$ are continuous and nonnegative functions for $ \tau, \sigma, \varsigma \in [0,\tau]$ and are nondecreasing with respect to $\tau,\, \beta_{k}(\tau)\,\,(k\in \mathbb{N})$ are nondecreasing for $\tau\in [0,\tau].$ Then, for $\tau\in [0,\tau]$, the following inequality is valid:
\begin{small}
\begin{align}\label{v2}
u(\tau)&\leq a(\tau)\underset{0<\tau_{k}<\tau}{\prod}\left( 1+\beta_{k}(\tau)\right) \exp \left( \int_0^\tau b(\tau, \sigma, )d\sigma+\int_0^\tau\int_0^\sigma k_{1}(\tau, \sigma, \varsigma)d\varsigma d\sigma\right.\nonumber\\
&\left.\qquad\qquad\qquad\qquad\qquad+\int_0^\tau\int_0^b k_{2}(\tau, \sigma, \varsigma)d\varsigma d\sigma\right).
\end{align}
\end{small}
\end{theorem}

\begin{proof}

Denote the right hand side of following inequality \eqref{v1} by $\mathscr V(\tau)$
\begin{small}
\begin{align*}
\mathscr V(\tau)&=a(\tau)+\int_0^\tau b(\tau,\sigma)u(\sigma)d\sigma+\int_0^\tau\left( \int_0^\sigma k_{1}(\tau, \sigma, \varsigma) u(\varsigma) d\varsigma \right)d\sigma+ \int_0^\tau\left( \int_0^b k_{2}(\tau, \sigma, \varsigma) u(\varsigma) d\varsigma \right)d\sigma\nonumber\\
&\qquad+\underset{0<\tau_{k}<\tau}{\sum} ~\beta_{k}(\tau)u(\tau_{k}).
\end{align*}
\end{small}

Then the function $\mathscr V(\tau)\in PC\left([0,b],\mathbb{R}_{+}\right)$ is nondecreasing, $u(\tau)\leq \mathscr V(\tau)$ and 
\begin{small}
\begin{align}\label{v3}
\mathscr V(\tau)&\leq a(\tau)+\int_0^\tau b(\tau,\sigma)\mathscr V(\sigma)d\sigma+\int_0^\tau\left( \int_0^\sigma k_{1}(\tau, \sigma, \varsigma) \mathscr V(\varsigma) d\varsigma \right)d\sigma+ \int_0^\tau\left( \int_0^b k_{2}(\tau, \sigma, \varsigma) \mathscr V(\varsigma) d\varsigma \right)d\sigma\nonumber\\
&\qquad+\underset{0<\tau_{k}<\tau}{\sum} ~\beta_{k}(\tau)\mathscr V(\tau_{k})\nonumber\\
&\leq a(\tau)+\int_0^\tau b(\tau,\sigma)\mathscr V(\sigma)d\sigma+\int_0^\tau\left( \int_0^\sigma k_{1}(\tau, \sigma, \varsigma) \mathscr V(\sigma) d\varsigma \right)d\sigma+ \int_0^\tau\left( \int_0^b k_{2}(\tau, \sigma, \varsigma) \mathscr V(\sigma) d\varsigma \right)d\sigma\nonumber\\
&\qquad+\underset{0<\tau_{k}<\tau}{\sum} ~\beta_{k}(\tau) \mathscr V(\tau_{k})\nonumber\\
&= a(\tau)+\int_0^\tau\left[ b(\tau,\sigma)+ \int_0^s k_{1}(\tau, \sigma, \varsigma)  d\varsigma+  \int_0^b k_{2}(\tau, \sigma, \varsigma)  d\varsigma \right]\mathscr V(\sigma) d\sigma +\underset{0<t_{k}<t}{\sum} ~\beta_{k}(\tau)\mathscr V(\tau_{k}).
\end{align}
\end{small}

Applying Lemma \ref{L1} to the inequality \eqref{v3}, with
\begin{small}
\begin{align*}
& u(\tau)= \mathscr V(\tau),\\
& \tilde{a}(\tau)=a(\tau),\\
& g(\tau,\sigma)=b(\tau,\sigma)+ \int_0^\sigma k_{1}(\tau, \sigma, \varsigma)  d\varsigma+  \int_0^b k_{2}(\tau, \sigma, \varsigma)  d\varsigma,  \\
&\tilde{\beta}_{k} (\tau)= \beta_{k}(\tau),
\end{align*}
\end{small}
we obtain
\begin{small}
\begin{align}\label{F9}
\mathscr V(\tau)\leq a(\tau)\underset{0<\tau_{k}<\tau}{\prod}\left( 1+\beta_{k}(\tau)\right) \exp \left( \int_0^\tau b(\tau,\sigma)d\sigma+\int_0^\tau\int_0^\sigma k_{1}(\tau, \sigma, \varsigma)  d\varsigma\,d\sigma+\int_0^\tau\int_0^b k_{2}(\tau, \sigma, \varsigma)  d\varsigma \,d\sigma\right).
\end{align}
\end{small}
From the inequality \eqref{F9}, we obtain the desired inequality \eqref{v3}.
\end{proof}


\section{Applications of mixed version of integral inequality}
In this section, we give the application of the mixed version of integral inequality to examine the continuous dependence of solutions on initial data and functions involved in equation. Further, we analyse the dependency by means of concept of $\epsilon$-approximate solutions and utilizing mixed version of integral inequality.   
\begin{theorem}\label{th5}
Suppose that the  hypothesis {\rm (H1)-(H3)} and {\rm (A4)} are satisfied. Let $w$ and $v$ are the mild solutions of \eqref{F1}-\eqref{F2} and \eqref{F5}-\eqref{F6} respectively. Then
\begin{small}
\begin{align}\label{0004}
&\left\| w-v\right\|_{\mathcal{PB}}
\leq \left(\mathscr M\left\|w_{0}-\widehat{w}_{0}\right\|+ b\,\mathscr M \mu +n\,\mathscr M\,\eta\right)\times \nonumber\\
&\qquad\qquad\prod_{k=1}^{n} (1+\mathscr M\, L_{I_{k}}) \exp\left(\mathscr M\,L_{G}\,b 
+\mathscr M\,L_{G}\,L_{F_{1}} \frac{b^{2}}{2}+\mathscr M\,L_{G}\,L_{F_{2}}\,b^{2} \right). 
\end{align}
\end{small}
\end{theorem}
\begin{proof}

Let $w$ and $v$ be the mild solution  of \eqref{F1}-\eqref{0F3} and  \eqref{F5}-\eqref{0F7} respectively. Then utilizing hypothesis (H1), (H2), (H3) and (A4), we get
\begin{small}
\begin{align} \label{F11}
&\left\| w(\tau)-v(\tau)\right\|\nonumber\\
&\leq \left\| \mathscr T(\tau)\right\|_{B(X)} \left\| w_{0}-\widehat{w}_{0}\right\| +\int_0^\tau \left\| \mathscr T(\tau-\sigma)\right\|_{B(X)}  \left\|  G\left(\sigma, w(\sigma), \int_0^\sigma F_{1}(\sigma,\varsigma,w(\varsigma))d\varsigma,\right.\right.\nonumber\\
&\qquad \left.\left.\int_0^b  F_{2}(\sigma,\varsigma,w(\varsigma))d\varsigma \right)-\widehat{G}\left(\sigma, v(\sigma), \int_0^\sigma \widehat{F}_1(\sigma,\varsigma,v(\varsigma))d\varsigma,\int_0^b  \widehat{F}_2(\sigma,\varsigma,v(\varsigma))d\varsigma \right)\right\| d\sigma \nonumber\\
& \qquad +\underset{0<\tau_{k}<\tau}{\sum} \left\| \mathscr T(\tau-\tau_{k})\right\|_{B(X)}  \left\| I_{k}\left(w(\tau_{k})\right)-\hat{I}_{k}(v(\tau_{k}))\right\|\nonumber\\
& \leq \left\| \mathscr T(\tau)\right\|_{B(X)} \left\| w_{0}-\widehat{w}_{0}\right\| +\int_0^\tau \left\| \mathscr T(\tau-\sigma)\right\|_{B(X)}  \left\|  G\left(\sigma, w(\sigma), \int_0^\sigma F_{1}(\sigma,\varsigma,w(\varsigma))d\varsigma,\right.\right.\nonumber\\
&\qquad \left.\left.\int_0^b  F_{2}(\sigma,\varsigma,w(\varsigma))d\varsigma \right)-G\left(\sigma, v(\sigma), \int_0^\sigma F_1(\sigma,\varsigma,v(\varsigma))d\varsigma,\int_0^b  F_2(\sigma,\varsigma,v(\varsigma))d\varsigma \right)\right\| d\sigma \nonumber\\
&\qquad+\int_0^\tau \left\| \mathscr T(\tau-\sigma)\right\|_{B(X)}  \left\|  G\left(\sigma, v(\sigma), \int_0^\sigma F_{1}(\sigma,\varsigma,v(\varsigma))d\varsigma,\right.\right.\nonumber\\
&\qquad \left.\left.\int_0^b  F_{2}(\sigma,\varsigma,v(\varsigma))d\varsigma \right)-\widehat{G}\left(\sigma, v(\sigma), \int_0^\sigma \widehat{F}_1(\sigma,\varsigma,v(\varsigma))d\varsigma,\int_0^b  \widehat{F}_2(\sigma,\varsigma,v(\varsigma))d\varsigma \right)\right\| d\sigma \nonumber\\
& \qquad \quad +\underset{0<\tau_{k}<\tau}{\sum} \left\| \mathscr T(\tau-\tau_{k})\right\|_{B(X)}  \left\| I_{k}\left(w(\tau_{k})\right)-I_{k}(v(\tau_{k}))\right\|\nonumber \\
& \qquad \qquad +\underset{0<\tau_{k}<\tau}{\sum} \left\| \mathscr T(\tau-\tau_{k})\right\|_{B(X)}  \left\| I_{k}\left(v(\tau_{k})\right)-\hat{I}_{k}(v(\tau_{k}))\right\|\nonumber\\
& \leq\mathscr M\left\|w_{0}-\widehat{w}_{0} \right\|+ \tau\mathscr M \mu +n\,\mathscr M\,\eta + \int_0^\tau \mathscr M L_{G}\left\| w(\sigma)-v(\sigma)\right\| d\sigma +\int_0^\tau \int_0^\sigma \mathscr ML_{G}L_{F_{1}}\nonumber\\
& \qquad   \left\| w(\varsigma)-v(\varsigma)\right\| d\varsigma d\sigma + \int_0^\tau \int_0^b \mathscr ML_{G}L_{F_{2}}\left\| w(\varsigma)-v(\varsigma)\right\| d\varsigma d\sigma+ \underset{0<\tau_{k}<\tau}{\sum}  \mathscr M\, L_{I_{k}}\left\| w(\tau_{k})-v(\tau_{k})\right\|.
\end{align}
\end{small}

Applying impulsive inequality from the  Theorem \ref{Th3} to \eqref{F11}
with
\begin{align*} 
& u(\tau)= \left\| w(\tau)-v(\tau)\right\|,\\
&a(\tau)=\mathscr M\left\|w_{0}-\widehat{w}_{0} \right\|+ \tau\mathscr M \mu +n\,\mathscr M\,\eta,\\
& b(\tau,\sigma)=\mathscr M\, L_{G},\,\\ 
& k_{1}(\tau,\sigma,\varsigma)=\mathscr M\, L_{G}\,L_{F_{1}},\\
& k_{2}(\tau, \sigma, \varsigma)=\mathscr M\, L_{G}\,L_{F_{2}},\\
&\beta_{k} (\tau)= \mathscr M L_{I_{k}},
\end{align*}
we obtain
\begin{small}
\begin{align*} 
&\left\| w(\tau)-v(\tau)\right\|\nonumber\\
& \leq \left(\mathscr M\left\|w_{0}-\widehat{w}_{0} \right\|+ \tau\mathscr M \mu +n\,\mathscr M\,\eta\right) \prod_{0<\tau_{k}<0}^{} (1+\mathscr M\, L_{I_{k}}) \exp\left(\mathscr M\,L_{G}\,b 
+\mathscr M\,L_{G}\,L_{F_{1}} \frac{b^{2}}{2}+\mathscr M\,L_{G}\,L_{F_{2}}\,b^{2} \right)\\
& \leq \left(\mathscr M\left\|w_{0}-\widehat{w}_{0}\right\|+ b\,\mathscr M\, \mu +n\,\mathscr M\,\eta\right) \prod_{k=1}^{n} 1+\mathscr M\, L_{I_{k}}) \exp\left(\mathscr M\,L_{G}\,b 
+\mathscr M\,L_{G}\,L_{F_{1}} \frac{b^{2}}{2}+\mathscr M\,L_{G}\,L_{F_{2}}\,b^{2} \right).
\end{align*}
\end{small}

Thus we get
\begin{small}
\begin{align*} 
\left\| w-v\right\|_{\mathcal{PC}}&=\underset{\tau\in J}{\sup}\left\lbrace \frac{\left\|w(\tau)-v(\tau) \right\|}{e^{\gamma \tau}} \right\rbrace \nonumber\\
& \leq \left(\mathscr M\left\|w_{0}-\widehat{w}_{0} \right\|+ b\,\mathscr M \mu +n\,\mathscr M\,\eta\right)\times\\ 
&\quad \quad\prod_{k=1}^{n} (1+\mathscr M\, L_{I_{k}}) \exp\left(\mathscr M\,L_{G}\,b 
+\mathscr M\,L_{G}\,L_{F_{1}} \frac{b^{2}}{2}+\mathscr M\,L_{G}\,L_{F_{2}}\,b^{2} \right),
\end{align*}
\end{small}
which is desired inequality \eqref{0004}.
\end{proof}

\begin{definition} For a given constant $\epsilon \geq 0$, a function $ w\in \Theta_{\mathcal{PB}}$ satisfying the inequality
\begin{small}
\begin{align*}
\left\| w'(\tau)-\mathscr Aw(\tau)-G\left(\tau, w(\tau), \int_0^\tau  F_{1}(\tau,\sigma,w(\sigma))d\sigma,\int_0^b  F_{2}(\tau,\sigma,w(\sigma))d\sigma\right)\right\|\leq \epsilon, ~ \tau\in J.
\end{align*}
\end{small}
subject to  $w(0)=w_{0}$ and $ \Delta w(\tau_{k})= I_{k}(w(\tau_{k})),~k=1,2,\cdots,n$, is called a $\epsilon$-approximate solution of the {\rm VFIIDE} \eqref{F1}.
\end{definition}
\begin{theorem}

Assume that (H1)-(H3) holds. If $w_{j}(\tau),\,(j=1,2)$ be $\epsilon_{j}$-approximate solutions of VFIIDE \eqref{F1} corresponding to  $w^{j}(0)=w^{j}_{0} \in X,\, \Delta w_{j}(\tau_{k})= I_{k}(w_{j}(\tau_{k})) \in X,~k=1,2,\cdots,n$ respectively. Then
\begin{small}
\begin{align}\label{F12}
\left\| w_{1}-w_{2}\right\|_{\mathcal{PB}}
&\leq\left\lbrace \left( \epsilon_{1}+\epsilon_{2}\right)\,\mathscr M (b+n)+ \mathscr M \left\| w^{1}_{0}-w^{2}_{0}\right\|\right\rbrace \times\nonumber\\
& \qquad\qquad \prod_{k=1}^{n} (1+\mathscr M\, L_{I_{k}}) \exp\left(\mathscr M\,L_{G}\,b 
+\mathscr M\,L_{G}\,L_{F_{1}} \frac{b^{2}}{2}+\mathscr M\,L_{G}\,L_{F_{2}}\,b^{2} \right).
\end{align}
\end{small}
\end{theorem}
\begin{proof}

Let $w_{j}(\tau),\,(j=1,2)$ be $\epsilon_{j}$-approximate solutions of VFIIDE \eqref{F1} corresponding to  $w^{j}(0)=w^{j}_{0} \in X,\, \Delta w_{j}(\tau_{k})= I_{k}(w_{j}(\tau_{k})) \in X,~k=1,2,\cdots,n$ respectively. Then we have
\begin{small}
\begin{align}\label{4}
\left\| w_{j}'(\tau)-\mathscr Aw_{j}(\tau)-G\left(\tau, w_{j}(\tau), \int_0^\tau F_1(\tau,\sigma,w_{j}(\sigma))d\sigma,\int_0^b  F_2(\tau,\sigma,w_{j}(\sigma))d\sigma\right)\right\|\leq \epsilon_{j}, ~ \tau\in J.
\end{align}
\end{small}

Then there exist $P_{w_{j}}\in PC(I,X)$ and a sequence $(P_{w_{j}})_{k}$ (dependence on $w_{j}$ ) such that
\begin{enumerate}
\item[(i)] $\left\| P_{w_{j}}(\tau)\right\|\leq \epsilon_{j},\,\tau \in J \qquad \text{and}\qquad \left\| (P_{w_{j}})_{k}\right\|\leq \epsilon_{j},\,\,k=1,2,\cdots,n.$
\item[(ii)] $ w_{j}'(\tau)=\mathscr Aw_{j}(\tau)+G\left(\tau, w_{j}(\tau), \int_0^\tau F_1(\tau,\sigma,w_{j}(\sigma))d\sigma,\int_0^b  F_2(\tau,\sigma,w_{j}(\sigma))d\sigma\right)+ P_{w_{j}}(\tau),\,\tau\in J.$
\item[(iii)] $\Delta w_{j}(\tau_{k})= I_{k}(w_{j}(\tau_{k}))+( P_{w_{j}})_{k},\,k=1,2,\cdots,n.$
\end{enumerate}

This gives
\begin{small}
\begin{align}\label{F13}
& w_{j}(\tau)=\mathscr T(\tau)w^{j}_{0}+\int_0^\tau \mathscr T(\tau-\sigma)\left[  G\left(\sigma, w_{j}(\sigma), \int_0^\sigma  F_{1}(\sigma,\varsigma,w_{j}(\varsigma))d\varsigma,\int_0^b  F_{2}(\sigma,\varsigma,w_{j}(\varsigma))d\varsigma \right)+ P_{w_{j}}(\sigma)\right] d\sigma\nonumber\\
&  \qquad +\sum_{k=1}^{n} \mathscr T(\tau-\tau_{k}) \left[ I_{k}\left(w_{j}(\tau_{k})\right)+(P_{w_{j}})_{k}\right]\nonumber\\
& \left\| w_{j}(\tau)-\mathscr T(\tau)w^{j}_{0}-\int_0^\tau \mathscr T(\tau-\sigma) G\left(\sigma, w_{j}(\sigma), \int_0^\sigma  F_{1}(\sigma,\varsigma,w_{j}(\varsigma))d\varsigma,\int_0^b  F_{2}(\sigma,\varsigma,w_{j}(\varsigma))d\varsigma \right) d\sigma\right.\nonumber\\
& \left.\qquad-\sum_{k=1}^{n} \mathscr T(\tau-\tau_{k}) I_{k}\left(w_{j}(\tau_{k})\right)\right\|\nonumber\\
&\leq\int_0^\tau \left\| \mathscr T(\tau-\sigma)\right\|  \left\| P_{w_{j}}(\sigma)\right\|  d\sigma
+\sum_{k=1}^{n} \left\| \mathscr T(\tau-\tau_{k})\right\| \left\| (P_{w_{j}})_{k}\right\|\nonumber\\
&\leq \tau \mathscr M \epsilon_{j}+\mathscr M\,n\,\epsilon_{j}= \epsilon_{j}\,\mathscr M (\tau+n),\,\, j=1,2,\,\, \tau\in J.
\end{align}
\end{small}
Therefore from \eqref{F13} we have
\begin{small}
\begin{align}\label{5}
&\left( \epsilon_{1}+\epsilon_{2}\right)\,\mathscr M (\tau+n)\nonumber\\
&\geq\left\| w_{1}(\tau)-\mathscr T(\tau)w^{1}_{0}-\int_0^\tau \mathscr T(\tau-\sigma) G\left(\sigma, w_{1}(\sigma), \int_0^\sigma  F_{1}(\sigma,\varsigma,w_{1}(\varsigma))d\varsigma,\int_0^b  F_{2}(\sigma,\varsigma,w_{1}(\varsigma))d\varsigma \right) d\sigma\right.\nonumber\\
& \qquad\left.-\sum_{k=1}^{n} \mathscr T(\tau-\tau_{k}) I_{k}\left(w_{1}(\tau_{k})\right)\right\|+\left\| w_{2}(\tau)-\mathscr T(\tau)w^{2}_{0}-\int_0^\tau \mathscr T(\tau-\sigma) G\left(\sigma, w_{2}(\sigma),\right.\right.\nonumber\\
&\left.\left.\qquad\quad \int_0^\sigma  F_{1}(\sigma,\varsigma,w_{2}(\varsigma))d\varsigma,\int_0^b  F_{2}(\sigma,\varsigma,w_{2}(\varsigma))d\varsigma \right) d\sigma-\sum_{k=1}^{n} \mathscr T(\tau-\tau_{k}) I_{k}\left(w_{2}(\tau_{k})\right)\right\|.
\end{align}
\end{small}

As we know  for  any $\xi_{1},\xi_{2} \in X, \left\|\xi_{1}- \xi_{2}\right\|\leq \|\xi_{1}\|+ \|\xi_{2}\|$  and $|\|\xi_{1}\|- \|\xi_{2}\||\leq \left\|\xi_{1}- \xi_{2}\right\|$. Using this in Eq.\eqref{5}, we get
\begin{small}
\begin{align}\label{6}
&\left( \epsilon_{1}+\epsilon_{2}\right)\,\mathscr M (\tau+n)\nonumber\\
& \geq\left\| \left\lbrace w_{1}(\tau)-\mathscr T(\tau)w^{1}_{0}-\int_0^\tau \mathscr T(\tau-\sigma) G\left(\sigma, w_{1}(\sigma), \int_0^\sigma F_{1}(\sigma,\varsigma,w_{1}(\varsigma))d\varsigma,\int_0^b  F_{2}(\sigma,\varsigma,w_{1}(\varsigma))d\varsigma \right)d\sigma\right.\right.\nonumber\\
&\left.\left. \quad -\sum_{k=1}^{n} \mathscr T(\tau-\tau_{k}) I_{k}\left(w_{1}(\tau_{k})\right)\right\rbrace-\left\lbrace  w_{2}(\tau)-\mathscr T(\tau)w^{2}_{0}-\int_0^\tau \mathscr T(\tau-\sigma) G\left(\sigma, w_{2}(\sigma)\right. \right.\right.\nonumber\\
&\left.\left.\left. \quad \int_0^\sigma  F_{1}(\sigma,\varsigma,w_{2}(\varsigma))d\varsigma,\int_0^b   F_{2}(\sigma,\varsigma,w_{2}(\varsigma))d\varsigma \right)d\sigma-\sum_{k=1}^{n} \mathscr T(\tau-\tau_{k})I_{k} \left( w_{2}(\tau_{k})\right)\right\rbrace  \right\|\nonumber\\
& =\left\|\left[w_{1}(\tau)-w_{2}(\tau)\right]-\left( \mathscr T(\tau)\left[ w^{1}_{0}-w^{2}_{0}\right]\right)- \left\lbrace  \int_0^\tau \mathscr T(\tau-\sigma)\left[G\left(\sigma, w_{1}(\sigma), \int_0^\sigma  F_{1}(\sigma,\varsigma,w_{1}(\varsigma))d\varsigma,\right.\right. \right.\right.\nonumber\\
&\qquad\left.\left.\left.\left. \int_0^b  F_{2}(\sigma,\varsigma,w_{1}(\varsigma))d\sigma \right)- G\left(\sigma, w_{2}(\sigma)\int_0^\sigma  F_{1}(\sigma,\varsigma,w_{2}(\varsigma))d\varsigma,\int_0^b  F_{2}(\sigma,\varsigma,w_{2}(\varsigma))d\varsigma \right)\right] d\sigma \right\rbrace \right.\nonumber\\
&\left. \qquad-\left\lbrace \sum_{k=1}^{n}\mathscr T(\tau-\tau_{k}) \left[I_{k} \left( w_{1}(\tau_{k})\right)-I_{k} \left( w_{2}(\tau_{k})\right) \right] \right\rbrace \right\| \nonumber\\
&\geq \left\| w_{1}(\tau)-w_{2}(\tau)\right\|-\left\| \mathscr T(\tau)\left[ w^{1}_{0}-w^{2}_{0}\right]\right\|-\left\| \int_0^\tau \mathscr T(\tau-\sigma) \left[ G\left(\sigma, w_{1}(\sigma)\int_0^\sigma  F_{1}(\sigma,\varsigma,w_{1}(\varsigma))d\varsigma \right.\right.\right.\nonumber\\
&\qquad\left.\left.\left.\int_0^b  F_{2}(\sigma,\varsigma,w_{1}(\varsigma))d\varsigma \right)-  G\left(\sigma, w_{2}(\sigma)\int_0^\sigma F_{1}(\sigma,\varsigma,w_{2}(\varsigma))d\varsigma,\int_0^b  F_{2}(\sigma,\varsigma,w_{2}(\varsigma))d\varsigma \right)\right]d\sigma \right\|\nonumber\\
&\quad \qquad -\left\|\sum_{k=1}^{n}\mathscr T(\tau-\tau_{k})\left[ I_{k} \left( w_{1}(\tau_{k})\right)-I_{k} \left( w_{2}(\tau_{k})\right) \right]  \right\|.
\end{align}
\end{small}

In Eq.\eqref{6} can be written as 
\begin{small}
\begin{align} \label{7}
&\left\| w_{1}(\tau)-w_{2}(\tau)\right\|\nonumber\\
&\leq\left( \epsilon_{1}+\epsilon_{2}\right)\,\mathscr M (\tau+n)+\left\| \mathscr T(\tau)\left[ w^{1}_{0}-w^{2}_{0}\right]\right\|+\left\| \int_0^\tau \mathscr T(\tau-\sigma) \left[ G\left(\sigma, w_{1}(\sigma)\int_0^\sigma , F_{1}(\sigma,\sigma,w_{1}(\varsigma))d\varsigma\right.\right.\right.\nonumber\\
&\qquad\left.\left.\left.\int_0^b  F_{2}(\sigma,\varsigma,w_{1}(\varsigma))d\varsigma \right)-  G\left(\sigma, w_{2}(\sigma)\int_0^\sigma  F_{1}(\sigma,\varsigma,w_{2}(\varsigma))d\varsigma,\int_0^b  F_{2}(\sigma,\varsigma,w_{2}(\varsigma))d\varsigma \right)\right] d\sigma \right\|\nonumber\\
&\quad\qquad+\left\|\sum_{k=1}^{n}\mathscr T(\tau-\tau_{k})\left[ I_{k} \left( w_{1}(\tau_{k})\right)-I_{k} \left( w_{2}(\tau_{k})\right) \right]  \right\| .
\end{align}
\end{small}

 Using hypotheses (H1)-(H3)  and  let $\bb(\tau) = \left\| w_{1}(\tau)-w_{2}(\tau)\right\| $ in \eqref{7} we get,
\begin{small}
\begin{align} \label{8}
& \bb(\tau)\leq \left( \epsilon_{1}+\epsilon_{2}\right)\,\mathscr M (\tau+n)+ \mathscr M\left\| w^{1}_{0}-w^{2}_{0}\right\|+ \int_0^\tau \mathscr M\, L_{G}\bb(\sigma) d\sigma + \int_0^\tau \int_0^\sigma \mathscr M\,L_{G}\,L_{F_{1}}\bb(\varsigma) d\varsigma d\sigma\nonumber\\
& \qquad \qquad\int_0^\tau \int_0^b \mathscr M\,L_{G}\,L_{F_{2}}\bb(\varsigma) d\varsigma d\sigma+ \underset{0<\tau_{k}<\tau}{\sum} \mathscr M\, L_{I_{k}}\bb(\tau_{k}).
\end{align}
\end{small}

Applying inequality from  the Theorem \ref{Th3} to \eqref{8} with
\begin{align*}
&  u(\tau)=\bb(\tau),\\
& a(\tau)=\left( \epsilon_{1}+\epsilon_{2}\right)\,\mathscr M (\tau+n)+ \mathscr M\left\| w^{1}_{0}-w^{2}_{0}\right\|\\
& b(\tau,\sigma)=\mathscr M\, L_{G},\\ 
& k_{1}(\tau,\sigma,\varsigma)=\mathscr M\, L_{G}\,L_{F_{1}},\\
& k_{2}(\tau,\sigma,\varsigma)=\mathscr M\, L_{G}\,L_{F_{2}},\\
&\beta_{k} (\tau)= \mathscr M L_{I_{k}},
\end{align*}
we get
\begin{small}
\begin{align*}
\bb(\tau)&\leq\left\lbrace \left( \epsilon_{1}+\epsilon_{2}\right)\,\mathscr M (\tau+n)+ \mathscr M\left\| w^{1}_{0}-w^{2}_{0}\right\|\right\rbrace\\
&\qquad \prod_{0<\tau_{k}<0}^{} (1+\mathscr M\, L_{I_{k}}) \exp\left(\mathscr M\,L_{G}\,b 
+\mathscr M\,L_{G}\,L_{F_{1}} \frac{b^{2}}{2}+\mathscr M\,L_{G}\,L_{F_{2}}\,b^{2} \right)\\
&\leq\left\lbrace \left( \epsilon_{1}+\epsilon_{2}\right)\,\mathscr M (b+n)+ \mathscr M\left\| w^{1}_{0}-w^{2}_{0}\right\|\right\rbrace\\
&\qquad  \prod_{k=1}^{n} (1+\mathscr M\, L_{I_{k}}) \exp\left(\mathscr M\,L_{G}\,b 
+\mathscr M\,L_{G}\,L_{F_{1}} \frac{b^{2}}{2}+\mathscr M\,L_{G}\,L_{F_{2}}\,b^{2} \right)
\end{align*}
\end{small}

But $ \bb(\tau)=\left\| w_{1}(\tau)-w_{2}(\tau)\right\| $ we have
\begin{small}
\begin{align*}
\left\| w_{1}-w_{2}\right\|_{\mathcal{PC}}&=\underset{\tau\in J}{\sup}\left\lbrace \frac{\left\| w_{1}(\tau)-w_{2}(\tau) \right\|}{e^{\gamma \tau}} \right\rbrace \\
&\leq\left\lbrace \left( \epsilon_{1}+\epsilon_{2}\right)\,\mathscr M (b+n)+ \mathscr M\left\| w^{1}_{0}-w^{2}_{0}\right\|\right\rbrace\times\\
&\qquad  \prod_{k=1}^{n}(1+\mathscr M\, L_{I_{k}}) \exp\left(\mathscr M\,L_{G}\,b 
+\mathscr M\,L_{G}\,L_{F_{1}} \frac{b^{2}}{2}+\mathscr M\,L_{G}\,L_{F_{2}}\,b^{2} \right).
\end{align*}
\end{small}

Which gives the  required inequality \eqref{F12}.
\end{proof}
\begin{rem}
\begin{itemize}
\item{\rm (i)} Continuous dependence of solutions of \eqref{F1} on initial conditions obtained by putting  $\epsilon_{1}=\epsilon_{2}=0$ in  inequality \eqref{F12}. 
\item {\rm(ii)} Uniqueness of the solution of problem \eqref{F1}-\eqref{F2} obtained by putting $\epsilon_{1}=\epsilon_{2}=0$ and $w^{1}_{0}=w^{2}_{0}$ in inequality \eqref{F12}.
\end{itemize}
\end{rem}

\section{Concluding Remarks}
Existence and uniqueness of solution of  the
Volterra-Fredholm impulsive integrodifferential equations (VFIIDEs) have been successfully achieve, through Banach's fixed point theorem. We favourably achieve an interesting extension  that is a mixed version of integral inequality for piece-wise continuous functions. Further, continuous data dependence of solutions on initial condition and functions involved in right hand side of VFIIDEs obtained by two techniques  first via  Picard operator theory  and secondly  via mixed version of integral inequality for piece-wise continuous functions.
 
In view of obtaining continuous data dependence via  PO, the Eq.\eqref{f0} is holds when
\begin{small}
 $$L_{\mathcal{R}}=\frac{\mathscr M\,L_{G}}{\gamma}\left[ \left( 1-e^{-\gamma\,b}\right) \left( 1+ \frac{L_{F_{1}}}{\gamma}\right) + L_{F_{2}} \,b\, e^{\gamma\,b}\right] +\mathscr M\,e^{\gamma\,b} \sum_{k=1}^{n}L_{I_{k}}  <1.$$ 
 \end{small}
 This restriction have removed  when we obtained same results by  mixed version of integral inequality. 
 
 One can extend similar types of impulsive integral inequalities in  fractional case that can be applied to analyze various qualitative properties of fractional integrodifferential equations with impulse condition.
 
\section*{Acknowledgement}
 The first author is financially supported by UGC, New Delhi, India (Ref: F1-17.1/2017-18/RGNF-2017-18-SC-MAH-43083) and JVCS acknowledges the financial support of a PNPD-CAPES (nº88882.305834/2018-01) scholarship of the Postgraduate Program in Applied Mathematics of IMECC-Unicamp.

\end{document}